\documentclass[lettersize,journal]{IEEEtran}
\usepackage{amsmath,amsfonts,amsthm}
\usepackage{algorithmic}
\usepackage{algorithm}
\usepackage{array}
\usepackage[caption=false,font=normalsize,labelfont=sf,textfont=sf]{subfig}
\usepackage{textcomp}
\usepackage{stfloats}
\usepackage{url}
\usepackage{verbatim}
\usepackage{graphicx}
\usepackage{cite}

\theoremstyle{plain}
\newtheorem{thm}{Theorem}

\newtheorem{pro}[thm]{Problem}
\newtheorem{lem}[thm]{Lemma}

\theoremstyle{definition}
\newtheorem{defn}[thm]{Definition}
\newtheorem{ass}{Assumption}
\newtheorem{rmk}[thm]{Remark}

\begin{document}

\title{$H_2/H_{\infty}$ Control for Stochastic Differential Systems with Partial Observation}

\author{Changwang Xiao, Nan Yang, Qingxin Menq
\thanks{This work was supported by the National Natural Science Foundation of China (No. 12271158), the Natural Science Foundation of Zhejiang Province for Distinguished Young Scholar (No. LZ22A010005), the Natural Science Foundation of Huzhou City (No. 2023YZ46) and the Postgraduate Research and Innovation Project of Huzhou Normal University (No. 2024KYCX66). (Corresponding authors: Qinxin Meng.)}
\thanks{The authors are with Department of Mathematical Sciences, Huzhou Normal University, Zhejiang 313000, PR China (E-mail addresses: xiaocw@zjhu.edu.cn (C. Xiao), 2024389526@stu.zjhu.edu.cn (N. Yang), mqx@zjhu.edu.cn (Q. Meng)).}}

\markboth{Journal of xxxxxxxxxxxxx,,~Vol.~xx, No.~x, August~20xx}%
{Shell \MakeLowercase{\textit{et al.}}: A Sample Article Using IEEEtran.cls for IEEE Journals}

\IEEEpubid{0000--0000/00\$00.00~\copyright~2021 IEEE}

\maketitle

\begin{abstract}
	This paper investigates the $H_{2}/H_{\infty}$ control problem for linear stochastic differential systems under partial observation. Unlike existing studies that assume full state accessibility, we consider the scenario where the controller has access only to an observation process. The objective is to design a controller that balances the $H_2$ performance criterion with the $H_\infty$ robustness requirement under worst-case disturbances, formulated as a nonzero-sum differential game. Using the Kalman filtering method, we derive the corresponding optimal filtering equation. Furthermore, a Stochastic Bounded Real Lemma under the partial observation framework is established, providing necessary and sufficient conditions for the $H_\infty$ robustness constraint. We also show the connection between the existence of a Nash equilibrium and the solvability of the cross-coupled Riccati equations, and illustrate the effectiveness of the proposed approach through a numerical example involving an unmanned aerial vehicle (UAV).
\end{abstract}

\section{Introduction}\label{sec1}
\IEEEPARstart{I}{n} modern control theory and engineering practice,
achieving a balance between system performance and robustness against disturbances has always been a central pursuit. The $H_2$ control theory primarily focuses on minimizing the system's response to white noise or transient disturbances, thereby optimizing dynamic performance and energy efficiency. In contrast, $H_\infty$ control theory is designed to ensure robust stability and suppress the impact of exogenous disturbances (square-integrable signals) on the regulated output. Since real-world systems are often subject to both stochastic noises and deterministic uncertainties, a single performance index is often insufficient. Consequently, the mixed $H_2/H_\infty$ control problem, which seeks to minimize an $H_2$ performance criterion while guaranteeing a prescribed $H_\infty$ disturbance attenuation level, has garnered significant attention since its inception \cite{Bernstein1989, Doyle1989}.

The mixed $H_2/H_\infty$ control problem was initially formulated for deterministic systems. A landmark contribution by Limebeer et al. \cite{Limebeer1994} interpreted the problem as a two-player nonzero-sum differential game, where the controller and the disturbance minimize their respective cost functionals. In Limebeer's work, the Nash game-theoretic approach provides a clear framework for deriving solutions to the $H_2/H_\infty$ control problem via cross-coupled Riccati equations. Extending this framework to stochastic systems, Chen and Zhang \cite{Chen2004} investigated the stochastic $H_2/H_\infty$ control with state-dependent noise. Following this, significant progress has been made in stochastic $H_2/H_\infty$ control under the assumption of \textit{full state information}. Researchers have explored various system dynamics, including Markovian jump systems \cite{Zhang2007, Hou2013, Mei2020}, time-delay systems \cite{Gao2015, Li2018}. In recent studies, mean-field stochastic systems have been widely investigated. For instance, Wang et al. \cite{Wang2023} derived the Stochastic Bounded Real Lemma for mean-field systems and solved the $H_2/H_\infty$ problem via indefinite coupled Riccati differential equations. Furthermore, Fang et al. \cite{Fang2026} addressed the $H_2/H_\infty$ problem for linear stochastic systems with affine terms, providing explicit open-loop and closed-loop solvability conditions. In UAV control, Qi and Zhao \cite{Qi2025} proposed a hybrid $H_2/H_\infty$ control method based on LMI for quadrotor autonomous landing. Hui et al. \cite{Hui2024} applied robust $H_\infty$ control and two-stage cascaded MPC to quadrotor attitude control. Hasanlu and Siavashi \cite{Hasanlu2025} proposed a discrete $H_\infty$ control strategy for nonlinear quadrotor trajectory tracking, improving robust tracking performance.

\IEEEpubidadjcol

However, in many practical applications such as financial engineering, networked control systems, and large-scale population dynamics, full state observation is often unrealistic. Controllers usually have access only to an observation process providing partial information about the system state. This leads to \textit{stochastic optimal control under partial information}. A classical example is the Linear-Quadratic-Gaussian (LQG) problem. The standard solution relies on the separation principle \cite{Wonham1968, Davis1977}, which decouples the design of an optimal feedback controller into two independent subproblems: optimal state estimation (Kalman filtering) and optimal deterministic control (LQR). Under partial observations, the optimal control strategy retains a linear feedback structure, with the state replaced by its optimal estimate. Bensoussan \cite{Bensoussan1992} and James et al. \cite{James2002} studied stochastic control with partial observations using nonlinear filtering techniques. Huang et al. \cite{Huang2009} and Øksendal et al. \cite{Øksendal2010} investigated the stochastic maximum principle for partially observed systems. Wang et al. \cite{Wang2015} solved linear-quadratic problems for forward-backward stochastic differential equations under partial information. Recently, Sun and Xiong \cite{Sun2022} studied stochastic LQ control with partial observation, deriving an optimal control via the filtering process using orthogonal decomposition. Moon and Basar \cite{Moon2024} extended separation principles to mean-field stochastic systems. Other notable works include \cite{Wang2013, Wu2018, Li2019, Huang2020, Wang2021}.

Despite extensive literature on full-information $H_2/H_\infty$ control and partially observed LQ control, the \textit{stochastic $H_2/H_\infty$ control problem under partial information} remains largely unexplored. This problem presents unique theoretical challenges:

First, unlike the pure LQ problem in \cite{Sun2022, Moon2024}, the $H_2/H_\infty$ problem involves coupling between control and disturbance: the optimal control depends on the disturbance, while the worst-case disturbance depends on the control. This may invalidate the classical certainty equivalence principle. In this paper, we show that a modified separation-type structure can still be obtained via orthogonal decomposition.

Second, the partial information constraint prevents direct use of state-feedback Nash equilibrium strategies from \cite{Chen2004, Wang2023}, requiring the controller to rely on the filtration generated by the observation process.

Third, constructing a stabilizing observer-based controller satisfying both $H_\infty$ robustness and $H_2$ optimality requires solving filtering equations and coupled Riccati equations simultaneously, which is more complex than in the full-information case.

Before proceeding, we provide a motivating aerospace example: designing a quadrotor UAV controller that balances tracking accuracy and robustness against atmospheric disturbances. Following \cite{Shi2018,Qi2025,Hui2024,Hasanlu2025,Ming2023}, consider the linearized longitudinal dynamics. Let $x(t)$ denote the state vector. The system follows
\begin{equation*}
	\left\{
	\begin{aligned}
		&dx(t) = [Ax(t) + B_1 v(t) + B_2 u(t) + b(t)]\mathrm{d}t + C \mathrm{d}W(t),\\
		&x(0) = x_0,
	\end{aligned}
	\right.
\end{equation*}
where $u(t)$ is the control input (thrust and pitch torque), $v(t)$ the worst-case disturbance (e.g., wind gusts), $b(t)$ an affine term capturing constant mismatches, and $C \mathrm{d}W(t)$ the stochastic noise. Partial observations are modeled as
\begin{equation*}
	\left\{
	\begin{aligned}
		&dy(t) = [Ex(t) + \beta(t)]\mathrm{d}t + F \mathrm{d}W(t),\\
		&y(0) = 0,
	\end{aligned}
	\right.
\end{equation*}
with $\beta(t)$ simulating sensor biases. The control objectives are: ensure an $H_\infty$ attenuation level $\gamma>0$:
\begin{equation*}
	J_1(0, x_0; u(\cdot), v(\cdot)) := \mathbb{E} \int_{0}^{T} [\gamma^2 |v(t)|^2 - |z(t)|^2]\mathrm{d}t,
\end{equation*}
while minimizing the $H_2$ cost
\begin{equation*}
	J_2(0, x_0; u(\cdot), v(\cdot)) := \mathbb{E} \int_{0}^{T} |z(t)|^2 \mathrm{d}t,
\end{equation*}
with regulated output
\[
z(t) = \begin{pmatrix} Q(t)x(t)\\ N_1(t)u(t) \end{pmatrix}.
\]

To the best of our knowledge, no unified framework combines stochastic $H_2/H_\infty$ differential games with filtering for general linear stochastic systems. Existing partial information studies mainly address a single cost or pure $H_\infty$, rarely the mixed $H_2/H_\infty$ trade-off. Motivated by this, this paper investigates the stochastic $H_2/H_\infty$ problem under partial information. The main contributions are:

\begin{enumerate}
	\item Formulation of a partially observed $H_2/H_\infty$ problem as a nonzero-sum differential game subject to a filtration constraint.
	\item Integration of stochastic filtering with the differential game framework to establish a separation-like characterization of the Nash equilibrium under partial observation.
	\item Provision of sufficient conditions for the existence of the optimal control via coupled differential Riccati equations, extending results to partial information scenarios.
\end{enumerate}

The remainder of the paper is organized as follows: Section \ref{sec2} introduces notation and formulates the nonzero-sum differential game. Section \ref{sec3} derives the optimal filtering equations. Section \ref{sec4} presents a stochastic bounded real lemma under partial observation, providing necessary and sufficient $H_\infty$ robustness conditions. Section \ref{sec5} characterizes the closed-loop solvability via coupled Riccati equations. Section \ref{sec6} presents a numerical example, and Section \ref{sec7} concludes the paper.

\section{Preliminaries}\label{sec2}

\subsection{Notations}
In this paper, let $T>0$ be fixed, and let $(\Omega,\mathcal{F},\mathbb{P})$ be a complete probability space augmented by all the $\mathbb{P}$-null sets in $\mathcal{F}$. Let ${W=\left(W_{1}, \ldots, W_{r}\right)^{\top}}$ and ${\widetilde{W}=(\widetilde{W}_{1}, \ldots, \widetilde{W}_{p})^{\top}}$ be two mutually independent standard Brownian motions with values in $\mathbb{R}^r$ and $\mathbb{R}^p$, respectively, defined on $(\Omega, \mathcal{F}, \mathbb{P})$. Let \( \mathbb{F}:=\left\{\mathcal{F}_{t}\right\}_{0 \leq t \leq T} \) be the natural filtration generated by \( (W, \widetilde{W}) \) satisfying the usual conditions. Then \( (\Omega, \mathcal{F}, \mathbb{F}, \mathbb{P}) \) becomes a complete filtered probability space. Let $\mathbb{E}$ be the mathematical expectation with respect to $\mathbb{P}$. Denote by $\mathbb{R}^n$ the $n$-dimensional Euclidean space, $\mathbb{R}^{n\times m}$ the set of all real matrices, $\mathbb{S}^n$ the space of real symmetric $n \times n$ matrices, and $I_n$ the identity matrix in $\mathbb{R}^n$. Moreover, $A^\top$ and $A^{-1}$ denote the transpose and the inverse of a matrix $A$, respectively. For \( M, N \in \mathbb{R}^{n \times m} \), let \( \langle M, N\rangle = \operatorname{tr}(M^{\top} N) \) be the inner product and \( |M| = \sqrt{\operatorname{tr}(M^{\top} M)} \) the induced norm on \( \mathbb{R}^{n\times m} \).

For any Euclidean space $H$, we introduce the following spaces:
\begin{enumerate}
	\item $C([0,T];H)$: The space of continuous functions $X: [0,T] \rightarrow H$ satisfying
	\begin{equation*}
		\left\|X\right\|_{C([0,T];H)} := \sup_{t \in [0,T]}\left\|X(t)\right\|_{H} < \infty.
	\end{equation*}
	
	\item $L^2(\Omega,\mathcal{F},\mathbb{P};H)$: The space of square-integrable $\mathcal{F}$-measurable random variables $\xi: \Omega \rightarrow H$ satisfying
	\begin{equation*}
		\mathbb{E}[\|\xi\|_H^2] < \infty.
	\end{equation*}
	
	\item $L_{\mathbb{F}}^2([0,T];H)$: The space of $\mathbb{F}$-adapted stochastic processes $X: [0,T]\times \Omega \rightarrow H$ satisfying
	\begin{equation*}
		\left\|X\right\|_{L_{\mathbb{F}}^2([0,T];H)}^2 := \mathbb{E}\int_{0}^{T} \left\|X(t)\right\|_H^2 \mathrm{d}t < \infty.
	\end{equation*}
	
	\item $L_{\mathbb{F}}^2(\Omega;C([0,T];H))$: The space of $\mathbb{F}$-adapted continuous stochastic processes $X: [0,T] \times \Omega \rightarrow H$ satisfying
	\begin{equation*}
		\left\|X\right\|_{L_{\mathbb{F}}^2(\Omega;C([0,T];H))}^2 := \mathbb{E}\left[\sup_{t \in [0,T]} \left\|X(t)\right\|_H^2\right] < \infty.
	\end{equation*}
\end{enumerate}

\subsection{Problem Statement}
We consider the $\mathbb{R}^n$-valued linear stochastic differential equation (SDE) driven by two Brownian motions on $[0,T]$:
\begin{equation}\label{eq:2.1}
	\left\{
	\begin{aligned}
		\mathrm{d}x(t) =& \Big\{ A(t)x(t) + B_1(t)v(t) + B_2(t)u(t) + b(t) \Big\} \mathrm{d}t \\
		& + C(t)\mathrm{d}W(t) + D(t)\mathrm{d}\widetilde{W}(t),\\
		x(0)=&x_0,\quad 
		z(t)=
		\begin{pmatrix}
			Q(t)x(t)\\
			N_1(t)u(t)
		\end{pmatrix},\quad t \in [0,T].
	\end{aligned}
	\right.
\end{equation}
Here $x(t): [0,T] \times \Omega \rightarrow \mathbb{R}^n$ is the state process, which evolves under the control input $u(t): [0,T] \times \Omega \rightarrow \mathbb{R}^s$ and the exogenous disturbance $v(t): [0,T] \times \Omega \rightarrow \mathbb{R}^m$. The initial condition is $x_0 \in L^2(\Omega,\mathcal{F}_0,\mathbb{P};\mathbb{R}^n)$, and $z(t)$ represents the measured system output, which is related to the state $x(t)$ and the control input $u(t)$.

In many practical scenarios, due to the influence of environmental noise, only partial information on the state is available. Accordingly, the observation process is described by the following SDE:
\begin{equation}\label{eq:2.2}
	\left\{
	\begin{aligned}
		\mathrm{d}y(t) &= \Big\{ E(t)x(t) + \beta(t) \Big\}\mathrm{d}t + F(t)\mathrm{d}W(t),\\
		y(0) &= 0,
	\end{aligned}
	\right.
\end{equation}
where $y(t):[0,T]\times \Omega \rightarrow \mathbb{R}^r$ is the observation process and $y(0) = 0$ is the initial condition.

Throughout this paper we impose the following assumptions.

\begin{ass}\label{ass:2.1}
	The coefficient functions $A(\cdot) : [0,T] \rightarrow \mathbb{R}^{n \times n}$, $B_1(\cdot) : [0,T] \rightarrow \mathbb{R}^{n \times m}$, $B_2(\cdot) : [0,T] \rightarrow \mathbb{R}^{n \times s}$, $E(\cdot) : [0,T] \rightarrow \mathbb{R}^{r \times n}$ as well as the matrix-valued functions $Q(\cdot) : [0,T] \rightarrow \mathbb{R}^{n \times n}$ and $N_1(\cdot) : [0,T] \rightarrow \mathbb{R}^{s \times s}$ are deterministic and uniformly bounded on $[0,T]$.
\end{ass}

\begin{ass}\label{ass:2.2}
	The affine terms $b(\cdot) : [0,T] \rightarrow \mathbb{R}^{n}$, $C(\cdot) : [0,T] \rightarrow \mathbb{R}^{n \times r}$, $D(\cdot) : [0,T] \rightarrow \mathbb{R}^{n \times p}$, $\beta(\cdot) : [0,T] \rightarrow \mathbb{R}^{r}$, and $F(\cdot) : [0,T] \rightarrow \mathbb{R}^{r \times r}$ are deterministic and uniformly bounded on $[0,T]$.
\end{ass}

\begin{ass}\label{ass:2.3}
	The function $F(\cdot)$ is invertible for a.e. $t \in [0,T]$, and $F(\cdot)^{-1}$ is bounded on $[0,T]$.
\end{ass}

\begin{ass}\label{ass:2.4}
	The matrix $N_1(\cdot)$ satisfies
	\[
	N_1(t)^{\top}N_1(t) = I_s,\qquad \forall t \in [0,T].
	\]
\end{ass}

Let $\mathbb{Y} = \{\mathcal{Y}_t\}_{0 \le t \le T}$ be the natural filtration generated by the observation process $y(t)$, augmented in the usual way. Now, we define the sets of admissible control and admissible disturbance for state equation \eqref{eq:2.1} and observation equation \eqref{eq:2.2}:
\begin{equation*}
	\mathcal{U}^{\mathbb{Y}}:= L_{\mathbb{Y}}^2([0,T];\mathbb{R}^s),
	\qquad
	\mathcal{V}^{\mathbb{Y}}:= L_{\mathbb{Y}}^2([0,T];\mathbb{R}^m).
\end{equation*}

\begin{rmk}\label{rmk:2.1}
	Suppose that Assumptions \ref{ass:2.1}--\ref{ass:2.4} hold. Then, for any initial condition $x(0) = x_0 \in L^2(\Omega,\mathcal{F}_0,\mathbb{P};\mathbb{R}^n)$ and any control pair $(u(\cdot),v(\cdot)) \in \mathcal{U}^{\mathbb{Y}} \times \mathcal{V}^{\mathbb{Y}}$, the state equation \eqref{eq:2.1} and observation equation \eqref{eq:2.2} admit unique strong solutions $x(\cdot) := x^{x_0,u,v}(\cdot) \in L_{\mathbb{F}}^2(\Omega;C([0,T];\mathbb{R}^n))$ and $y(\cdot) \in L_{\mathbb{F}}^2(\Omega;C([0,T];\mathbb{R}^r))$, respectively. For convenience, we use the notation $x^{0,u,v}$ to represent the solution of system \eqref{eq:2.1} with zero initial value (i.e. $x(0) = 0$), and $x^{x_0,0,v}$ (resp. $x^{x_0,u,0}$) to denote the solution of system \eqref{eq:2.1} with $u(\cdot)=0$ (resp. $v(\cdot)=0$).
\end{rmk}

Different from the case of complete information, we now formally introduce the definition of the stochastic $H_2/H_\infty$ control problem with partial observation.

\begin{defn}\label{defn:2.2}
	Consider the controlled SDE \eqref{eq:2.1} on $[0,T]$, with a prescribed disturbance attenuation level $\gamma > 0$. Define the admissible closed-loop control strategy spaces as
	\begin{equation*}
		\begin{aligned}
			&\mathcal{N}^2[0,T] := C([0,T];\mathbb{R}^{s \times n}) \times \mathcal{U}^{\mathbb{Y}},\\
			&\mathcal{M}^2[0,T] := C([0,T];\mathbb{R}^{m \times n}) \times \mathcal{V}^{\mathbb{Y}}.
		\end{aligned}
	\end{equation*}
	
	We say that the $H_2/H_{\infty}$ problem under partial information admits an optimal closed-loop solution $(u^*(\cdot),v^*(\cdot))$, if there exists a pair of control strategies
	\begin{equation*}
		(U(\cdot),U_0(\cdot);V(\cdot),V_0(\cdot)) \in \mathcal{N}^2[0,T] \times \mathcal{M}^2[0,T],
	\end{equation*}
	such that
	\begin{equation}\label{eq:2.3}
		\begin{aligned}
			u^*(\cdot) &= U(\cdot)\widehat{x^*}(\cdot) + U_0(\cdot) \in \mathcal{U}^{\mathbb{Y}},\\
			v^*(\cdot) &= V(\cdot)\widehat{x^*}(\cdot) + V_0(\cdot) \in \mathcal{V}^{\mathbb{Y}},
		\end{aligned}
	\end{equation}
	where $x^*(\cdot)$ is the state process corresponding to the control pair $(u^*(\cdot),v^*(\cdot))$, and $\widehat{x^*}(\cdot)$ is given by
	\[
	\widehat{x^*}(t) = \mathbb{E}[x^*(t)\mid \mathcal{Y}_t], \qquad t \in [0,T].
	\]
	Moreover, the following two conditions hold.
	
	\medskip
	\noindent
	(1) ($H_{\infty}$ Robustness) Consider system \eqref{eq:2.1} with control input
	\[
	u(t)=U(t)\widehat{x}(t)+U_0(t),
	\]
	and define the linear operator $L : \mathcal{V}^{\mathbb{Y}} \rightarrow L_\mathbb{F}^2([0,T];\mathbb{R}^{n + s})$ by
	\[
	L(v)(t) := \widetilde{z}(t) =
	\begin{pmatrix}
		Q(t)\big[x^{x_0,U\widehat{x},v}(t) - x^{x_0,0,0}(t)\big]\\
		N_1(t)U(t)\big[\widehat{x}^{x_0,U\widehat{x},v}(t) - \widehat{x}^{x_0,0,0}(t)\big]
	\end{pmatrix},
	\]
	where the same affine term $U_0(\cdot)$ is used in the compared trajectories and hence is canceled in the difference. The operator norm $\|L\|$ satisfies
	\begin{equation*}
		\|L\| := \sup_{\substack{v(\cdot) \neq 0 \\ v(\cdot) \in \mathcal{V}^{\mathbb{Y}}}} \frac{\|\widetilde{z}\|_{[0,T]}}{\|v\|_{[0,T]}} < \gamma,
	\end{equation*}
	where
	\begin{equation*}
		\begin{aligned}
			\|v\|_{[0,T]} &:= \Big(\mathbb{E}\int_0^T |v(t)|^2 \mathrm{d}t\Big)^{\frac{1}{2}},\\
			\|\widetilde{z}\|_{[0,T]} &:= \Bigg(\mathbb{E}\int_0^T \Big\{|Q(t)\big[x^{x_0,U\widehat{x},v}(t) - x^{x_0,0,0}(t)\big]|^2 \\
			&\quad + |N_1(t)U(t)\big[\widehat{x}^{x_0,U\widehat{x},v}(t) - \widehat{x}^{x_0,0,0}(t)\big]|^2 \Big\} \mathrm{d}t \Bigg)^{\frac{1}{2}}.
		\end{aligned}
	\end{equation*}
	
	\medskip
	\noindent
	(2) ($H_2$ Optimality) Consider system \eqref{eq:2.1} with disturbance
	\[
	v(t) = V(t)\widehat{x}(t) + V_0(t).
	\]
	The closed-loop control laws $(u^*(\cdot),v^*(\cdot))$ defined in \eqref{eq:2.3} minimize the output energy
	\begin{equation*}
		\begin{aligned}
			\left\|z\right\|^2_{[0,T]}
			&= \mathbb{E}\int_{0}^{T} |z(t)|^2 \mathrm{d}t\\
			&= \mathbb{E}\int_{0}^{T} \Big[|Q(t)x^{x_0,u,V\widehat{x}+V_0}(t)|^2 + |N_1(t)u(t)|^2\Big] \mathrm{d}t.
		\end{aligned}
	\end{equation*}
\end{defn}

To ensure conditions (1) and (2) in Definition \ref{defn:2.2}, namely that the norm of the linear operator $L$ satisfies $\|L\| < \gamma$ for some $\gamma > 0$ and that the output energy is minimized, we define the following two performance functionals:
\begin{equation*}
	\begin{aligned}
		J_1(0,x_0;u(\cdot),v(\cdot)) &:= \mathbb{E}\int_{0}^{T} [\gamma^2 |v(t)|^2 - |z(t)|^2]\mathrm{d}t,\\
		J_2(0,x_0;u(\cdot),v(\cdot)) &:= \mathbb{E}\int_{0}^{T} |z(t)|^2 \mathrm{d}t.
	\end{aligned}
\end{equation*}

Minimizing $J_1$ enforces the $H_\infty$ robustness criterion by bounding the worst-case energy gain from the disturbance $v(\cdot)$ to the output $z(\cdot)$. Meanwhile, $J_2$ characterizes the $H_2$ performance as a measure of nominal output energy, and its minimization yields optimal regulation performance and energy efficiency. In conclusion, the $H_2/H_{\infty}$ control problem with partial observation is to find a pair of strategies $(u^*(\cdot),v^*(\cdot)) \in \mathcal{U}^{\mathbb{Y}} \times \mathcal{V}^{\mathbb{Y}}$ such that the control $u^*(\cdot)$ guarantees the prescribed $H_\infty$ disturbance attenuation level $\gamma$ against any disturbance $v(\cdot)\in\mathcal{V}^{\mathbb{Y}}$, while, under the corresponding worst-case disturbance $v^*(\cdot)$ minimizing $J_1(0,x_0;u^*(\cdot),v(\cdot))$, the control $u^*(\cdot)$ also minimizes the $H_2$ cost $J_2(0,x_0;u(\cdot),v^*(\cdot))$.

From a game-theoretic perspective, we can formulate this problem as the following nonzero-sum differential game.

\begin{pro}\label{pro:2.3}
	For state equation \eqref{eq:2.1}, we seek a pair of closed-loop strategies
	\begin{equation*}
		(U(\cdot),U_0(\cdot);V(\cdot),V_0(\cdot)) \in \mathcal{N}^2[0,T] \times \mathcal{M}^2[0,T],
	\end{equation*}
	such that the optimal pair
	\begin{equation}\label{eq:2.4}
		\begin{aligned}
			(u^*(\cdot),v^*(\cdot))
			&= (U(\cdot)\widehat{x^*}(\cdot) + U_0(\cdot),V(\cdot)\widehat{x^*}(\cdot) + V_0(\cdot)) \\
			&\in \mathcal{U}^{\mathbb{Y}} \times \mathcal{V}^{\mathbb{Y}},
		\end{aligned}
	\end{equation}
	and for any initial state $x_0$, the following two inequalities are satisfied:
	
	\medskip
	\noindent
	(1) for all $v(\cdot) \in \mathcal{V}^{\mathbb{Y}}$,
	\begin{equation}\label{eq:2.5}
		J_1(0,x_0;u^*(\cdot),v^*(\cdot)) \leq J_1(0,x_0;U(\cdot)\widehat{x}(\cdot) + U_0(\cdot),v(\cdot));
	\end{equation}
	
	\medskip
	\noindent
	(2) for all $u(\cdot) \in \mathcal{U}^{\mathbb{Y}}$,
	\begin{equation}\label{eq:2.6}
		J_2(0,x_0;u^*(\cdot),v^*(\cdot)) \leq J_2(0,x_0;u(\cdot),V(\cdot)\widehat{x}(\cdot) + V_0(\cdot)).
	\end{equation}
	
	Note that $\widehat{x^*}(\cdot)$ in \eqref{eq:2.4} is the corresponding optimal estimate process of system \eqref{eq:2.1} under $(u^*(\cdot),v^*(\cdot))$, and $\widehat{x}(\cdot)$ in \eqref{eq:2.5} is the estimate process of system \eqref{eq:2.1} under the control $u(\cdot) = U(\cdot)\widehat{x}(\cdot) + U_0(\cdot) \in \mathcal{U}^{\mathbb{Y}}$ and an arbitrary disturbance $v(\cdot) \in \mathcal{V}^{\mathbb{Y}}$. Similarly, $\widehat{x}(\cdot)$ in \eqref{eq:2.6} is the estimate process of system \eqref{eq:2.1} under the disturbance $v(\cdot) = V(\cdot)\widehat{x}(\cdot) + V_0(\cdot) \in \mathcal{V}^{\mathbb{Y}}$ and an arbitrary control input $u(\cdot) \in \mathcal{U}^{\mathbb{Y}}$.
	
	A closed-loop strategy tuple $(U(\cdot),U_0(\cdot);V(\cdot),V_0(\cdot))$ that satisfies the above conditions is called a closed-loop Nash equilibrium strategy, and $(u^*(\cdot),v^*(\cdot))$ is referred to as a closed-loop Nash equilibrium.
\end{pro}

\section{Filtering Equation}\label{sec3}
In this section, we focus on deriving the filtering equations, which lay the foundation for proving the Bounded Real Lemma and solving the $H_2/H_\infty$ problem in subsequent sections.

We consider the system with control input $u(\cdot) = U(\cdot)\widehat{x}(\cdot) + U_0(\cdot)$:

\begin{equation*}
	\left\{
	\begin{aligned}
		\mathrm{d}x(t) =& \Big\{ A(t)x(t) + B_2(t)U(t)\widehat{x}(t) + B_1(t)v(t) + b(t)\\
		& + B_2(t)U_0(t) \Big\} \mathrm{d}t + C(t)\mathrm{d}W(t) + D(t)\mathrm{d}\widetilde{W}(t),\\
		x(0)=&x_0,
	\end{aligned}
	\right.
\end{equation*}
by letting $B(t) := B_2(t)U(t), \sigma(t) := B_2(t)U_0(t) + b(t)$, it can be simplified to the following controlled SDE:
\begin{equation} \label{eq:3.1}
	\left\{
	\begin{aligned}
		\mathrm{d}x(t) =& \Big\{ A(t)x(t) + B(t)\widehat{x}(t) + B_1(t)v(t) +\sigma(t) \Big\} \mathrm{d}t\\
		& + C(t)\mathrm{d}W(t) + D(t)\mathrm{d}\widetilde{W}(t),\\
		x(0)=&x_0.
	\end{aligned}
	\right.
\end{equation}

This equation describes the evolution of the state $x(t)$ under the influence of stochastic disturbance $v(t)$, the state $x(t)$ and its conditional mathematical expectation $\widehat{x}(t)$, affine terms $\sigma(t)$, and system dynamics noise.

And the corresponding observation equation:
\begin{equation}\label{eq:3.2}
	\left\{
	\begin{aligned}
		& \mathrm{d}y(t) = \Big\{ E(t)x(t) + \beta(t) \Big\}\mathrm{d}t + F(t)\mathrm{d}W(t),\\
		& y(0) = 0.
	\end{aligned}
	\right.
\end{equation}

Let us define the optimal estimate process:
\[ 
\widehat{x}(t) := \mathbb{E}[x(t) \mid \mathcal{Y}_t],
\]
the estimate error and its covariance:
\[
\widetilde{x}(t) := x(t) - \widehat{x}(t),\Sigma(t) := \mathbb{E}[\widetilde{x}(t)\widetilde{x}(t)^{\top}].
\]

For the orthogonality principle, it follows that:
\begin{equation*}
	\begin{aligned}
		&\widehat{x}(t) \:\bot \:\widetilde{x}(t) \: \text{i.e.} \: \mathbb{E}[\langle \widehat{x}(t),\widetilde{x}(t) \rangle] = 0,\\
		&\mathbb{E}|x(t)|^2 = \mathbb{E}|\widehat{x}(t)|^2 + \mathbb{E}|\widetilde{x}(t)|^2,
	\end{aligned}
\end{equation*}
and $\widetilde{x}(t)$ is independent of $\mathcal{Y}_t$, for all $t \geq 0$.

Now we define the so-called innovation process:
\begin{equation}\label{eq:3.3}
	I(t) = y(t) - \int_0^t \Big[E(s)\widehat{x}(s) + \beta(s)\Big]\mathrm{d}s
\end{equation}
and the modified innovation process:
\begin{equation}\label{eq:3.4}
	\widehat{I}(t) = \int_{0}^{t} F(s)^{-1} \mathrm{d}I(s).
\end{equation}

To derive the filtering equation, we present the following three key lemmas (Lemma \ref{lem:3.1}-\ref{lem:3.3}) and omit the proofs of the first two, as their rigorous proofs can be directly found in \cite{Sun2022}.

\begin{lem}\label{lem:3.1}
	$\widehat{I}(t)$ is an $\mathbb{R}^r$-valued $\mathbb{Y}$-adapted Brownian motion and therefore a continuous, $\mathcal{Y}_t$-adapted, integrable martingale process.
\end{lem}

\begin{lem}\label{lem:3.2}
	Define the process as following:
	\begin{equation}\label{eq:3.5}
		\begin{aligned}
			J(t) = &\widehat{x}(t) - \widehat{x}(0)\\
			&- \int_{0}^{t} \Big[ (A(s) + B(s))\widehat{x}(s) + B_1(s)v(s) + \sigma(s) \Big]\mathrm{d}s,
		\end{aligned}
	\end{equation}
	then $J(t)$ is a continuous square-integrable $\mathbb{R}^{n}$-valued $\mathbb{Y}$-martingale with $J(0) = 0, a.s.$.
\end{lem}

\begin{lem}\label{lem:3.3}
	$J(t)$ can be written in the following form:
	\begin{equation*}
		J(t) = \int_{0}^{t} \Big[\Sigma(s)E(s)^{\top}H(s)^{-1} + C(s)F(s)^{-1}\Big]\mathrm{d}I(s),
	\end{equation*}
	where $\Sigma(\cdot) := \mathbb{E}[\widetilde{x}(\cdot)\widetilde{x}(\cdot)^{\top}], H(\cdot) := F(\cdot)F(\cdot)^{\top}$.
\end{lem}

\begin{proof}
	By the Fujisaki-Kallianpur-Kunita Theorem (see Theorem 8.3.1 in \cite{Kallianpur1980}), there exists a unique $\lambda(\cdot) \in L_{\mathbb{Y}}^2([0,T];\mathbb{R}^{n \times r})$ such that
	\begin{equation*}
		J(t) = \int_{0}^{t} \lambda(s)H(s)^{-1}\mathrm{d}I(s).
	\end{equation*}
	
	Next, let us determine $\lambda(\cdot)$. Introduce an $\mathbb{R}^{n}$-valued $\mathbb{Y}$-martingale process 
	\begin{equation*}
		\Psi(t) = \int_{0}^{t} \xi(s)H(s)^{-1}\mathrm{d}I(s),
	\end{equation*}
	where $\xi(\cdot) \in L_{\mathbb{Y}}^2([0,T];\mathbb{R}^{n \times r})$ is arbitrary. By It\^{o}'s formula and Lemma \ref{lem:3.1}, we have
	\begin{equation}\label{eq:3.6}
		\mathbb{E}[J(t)\Psi(t)^{\top}] = \mathbb{E}\int_{0}^{t} \lambda(s)H(s)^{-1}\xi(s)^{\top}\mathrm{d}s.
	\end{equation}
	
	By the definition of $J(t)$ in \eqref{eq:3.5} and the fact that $\mathbb{E}[\Psi(t)] = 0$, it follows that
	\begin{equation}\label{eq:3.7}
		\begin{aligned}
			\mathbb{E}[J(t)\Psi(t)^{\top}] =& \mathbb{E}[\widehat{x}(t)\Psi(t)^{\top}] - \mathbb{E}\left\{\int_{0}^{t} \Big[(A(s) + B(s))\widehat{x}(s) \right.\\
			&\left. + B_1(s)v(s) + \sigma(s)\Big]\Psi(t)^{\top}\mathrm{d}s\right\}\\
			=&\mathbb{E}[x(t)\Psi(t)^{\top}] - \mathbb{E}\left\{\int_{0}^{t} \Big[A(s)x(s) \right.\\
			&\left.+ B(s)\widehat{x}(s) + B_1(s)v(s) + \sigma(s)\Big]\Psi(t)^{\top}\mathrm{d}s\right\}.
		\end{aligned}
	\end{equation}
	
	The above equation holds for $s \leq t$, because
	\begin{equation*}
		\begin{aligned}
			\mathbb{E} \Big\{A(s)\widehat{x}(s)\Psi(t)^\top\Big\} &= \mathbb{E}\Big\{A(s) \mathbb{E}[\widehat{x}(s)\Psi(t)^{\top} \mid \mathcal{Y}_s]\Big\} \\
			&= \mathbb{E}\Big\{A(s)\widehat{x}(s)\Psi(s)^\top\Big\}\\
			& = \mathbb{E}\Big\{A(s)\mathbb{E}[x(s) \mid \mathcal{Y}_s]\Psi(s)^\top\Big\} \\
			&= \mathbb{E}\Big\{A(s)x(s)\Psi(s)^\top\Big\}.
		\end{aligned}
	\end{equation*}
	
	Since 
	\begin{equation*}
		\begin{aligned}
			\mathrm{d}\Psi(t) &= \xi(t)H(t)^{-1}\mathrm{d}I(t) \\
			&= \xi(t)H(t)^{-1}\{\mathrm{d}y(t) - [E(t)\widehat{x}(t) + \beta(t)]\mathrm{d}t\}\\
			& =\xi(t)H(t)^{-1}E(t)\widetilde{x}(t)\mathrm{d}t + \xi(t)H(t)^{-1}F(t)\mathrm{d}W(t),
		\end{aligned}
	\end{equation*}
	using It\^o's formula we have 
	\begin{equation}\label{eq:3.8}
		\begin{aligned}
			\mathbb{E}[x(t)\Psi(t)^{\top}] =& \mathbb{E}\bigg\{ \int_{0}^{t} [A(s)x(s) + B(s)\widehat{x}(s) + B_1(s)v(s) \\
			& + \sigma(s)]\Psi(s)^{\top}\mathrm{d}s \Big\}\\
			+& \mathbb{E} \left[\int_{0}^{t} x(s)[\xi(s) H(s)^{-1} E(s)\widetilde{x}(s)]^{\top}\mathrm{d}s\right]\\
			+& \mathbb{E}\left[\int_{0}^{t} C(s)[\xi(s)H(s)^{-1}F(s)]^{\top}\mathrm{d}s\right].
		\end{aligned}
	\end{equation}
	
	Due to $\xi(s) \in \mathcal{Y}_s, \mathbb{E}[\widetilde{x}(s)] = 0$, and $\widetilde{x}(s)$ is independent of $\mathcal{Y}_s$, then
	\begin{equation*}
		\begin{aligned}
			\mathbb{E}&\left[\int_{0}^{t} x(s)[\xi(s) H(s)^{-1} E(s)\widetilde{x}(s)]^{\top}\mathrm{d}s\right]\\
			=&\mathbb{E} \left[\int_{0}^{t} \widetilde{x}(s)\widetilde{x}(s)^\top E(s)^\top H(s)^{-1}\xi(s)^{\top}\mathrm{d}s\right] \\
			&+ \mathbb{E} \left[\int_{0}^{t} \widehat{x}(s)\widetilde{x}(s)^\top E(s)^\top H(s)^{-1}\xi(s)^{\top}\mathrm{d}s\right]\\
			=&\mathbb{E} \left[\int_{0}^{t} \Sigma(s)E(s)^\top H(s)^{-1}\xi(s)^{\top}\mathrm{d}s\right].
		\end{aligned}
	\end{equation*}
	
	Substituting the above equation into \eqref{eq:3.8} and in accordance with \eqref{eq:3.7} yields 
	\begin{equation}\label{eq:3.9}
		\begin{aligned}
			\mathbb{E}[J(t)\Psi(t)^{\top}] = \mathbb{E}\bigg[\int_0^t [\Sigma(s)E(s)^\top + C(s)F(s)^\top]\\
			\cdot H(s)^{-1}\xi(s)^\top\mathrm{d}s\Big].
		\end{aligned}
	\end{equation}
	
	Since $\xi(\cdot) \in L_{\mathbb{Y}}^2([0,T];\mathbb{R}^{n \times r})$ is arbitrary, by comparing \eqref{eq:3.6} and \eqref{eq:3.9}, we have obtained the expression for $\lambda(\cdot)$. This completes the proof.
\end{proof}

With the above preparations completed, we shall now proceed to derive the filtering equation. By the definition of $J(t)$ in \eqref{eq:3.5} and Lemma \ref{lem:3.3}, it follows that 
\begin{equation*}
	\begin{aligned}
		\widehat{x}(t) = &\widehat{x}_0 + \int_{0}^{t} \Big[[A(s)+B(s)]\widehat{x}(s) + B_1(s)v(s) + \sigma(s)\Big]\mathrm{d}s\\
		& + \int_{0}^{t} \Big[\Sigma(s)E(s)^\top H(s)^{-1} + C(s)F(s)^{-1}\Big]\mathrm{d}I(s),
	\end{aligned}
\end{equation*}
where $\Sigma(\cdot)$ is the covariance of estimate error $\widetilde{x}(\cdot)$, and $I(\cdot)$ is the innovation process defined by 
\begin{equation*}
	I(t) = y(t) - \int_{0}^{t} \Big[E(s)\widehat{x}(s) + \beta(s)\Big]\mathrm{d}s, \quad 0\leq t \leq T.
\end{equation*}

Expressed in the form of a differential equation, it is
\begin{equation}\label{eq:3.10}
	\left\{
	\begin{aligned}
		&\mathrm{d}\widehat{x}(t) = \Big\{[A(t) + B(t)]\widehat{x}(t) + B_1(t)v(t) + \sigma(t)\Big\}\,\mathrm{d}t \\
		&\quad\qquad + \Big[\Sigma(t)E(t)^{\top}H(t)^{-1} + C(t)F(t)^{-1}\Big] \\
		&\quad\qquad \cdot \Big\{\mathrm{d}y(t) - [E(t)\widehat{x}(t) + \beta(t)]\mathrm{d}t\Big\},\\
		&\widehat{x}(0) = \widehat{x}_0,
	\end{aligned}
	\right.
\end{equation}
equation \eqref{eq:3.10} is the filtering equation of \eqref{eq:3.1} we desire.

By subtracting $\mathrm{d}\widehat{x}$ from $\mathrm{d}x$, the estimate error $\widetilde{x}$ evolves
\begin{equation}\label{eq:3.11}
	\left\{
	\begin{aligned}
		\mathrm{d}\widetilde{x}(t) =& \Big\{A(t) - \Sigma(t)E(t)^{\top}H(t)^{-1}E(t) \\
		&- C(t)F(t)^{-1}E(t)\Big\}\widetilde{x}(t)\mathrm{d}t\\
		-& \Big\{\Sigma(t)E(t)^{\top}F(t)^{\top,-1}\Big\}\mathrm{d}W(t) + D(t)\mathrm{d}\widetilde{W}(t),\\
		\widetilde{x}(0) = &\widetilde{x}_0.
	\end{aligned}
	\right.
\end{equation}

Finally, we further need to derive the explicit equation satisfied by $\Sigma$. By It\^o's formula and noting that $\mathbb{E}[\widetilde{x}_0{\widetilde{x}_0}^{\top}] = \Sigma_0$, it follows that:
\begin{equation*}
	\begin{aligned}
		\Sigma(t) = & \Sigma_0 + \mathbb{E}\int_{0}^{t} \Big[A(s) - \Sigma(s)E(s)^{\top}H(s)^{-1}E(s) \\
		&\quad - C(s)F(s)^{-1}E(s)\Big]\widetilde{x}(s)\widetilde{x}(s)^{\top}\mathrm{d}s\\
		& + \mathbb{E}\int_{0}^{t} \widetilde{x}(s)\widetilde{x}(s)^{\top}\Big[A(s) - \Sigma(s)E(s)^{\top}H(s)^{-1}E(s) \\
		&\quad - C(s)F(s)^{-1}E(s)\Big]^{\top}\mathrm{d}s\\
		& + \mathbb{E}\int_{0}^{t} \Big[\Sigma(s)E(s)^{\top}H(s)^{-1}F(s)\Big]\\
		& \quad\cdot \Big[\Sigma(s)E(s)^{\top}H(s)^{-1}F(s)\Big]^{\top}\mathrm{d}s \\
		& + \mathbb{E}\int_{0}^{t} D(s)D(s)^{\top}\mathrm{d}s\\
		= & \Sigma_0 + \int_{0}^{t} \Big\{ [A(s) - C(s)F(s)^{-1}E(s)]\Sigma(s) \\
		&\quad + \Sigma(s)[A(s) - C(s)F(s)^{-1}E(s)]^{\top}\\
		&\quad - \Sigma(s)E(s)^{\top}H(s)^{-1}E(s)\Sigma(s) + D(s)D(s)^{\top}\Big\}\mathrm{d}s.
	\end{aligned}
\end{equation*}

That is, $\Sigma$ satisfies the equation:
\begin{equation}\label{eq:3.12}
	\begin{aligned}
		\dot{\Sigma}(t) &= A(t)\Sigma(t) + \Sigma(t) A(t)^{\top} + C(t)C(t)^{\top} \\
		&\quad + D(t)D(t)^{\top} - K(t)H(t)K(t)^{\top},
	\end{aligned}
\end{equation}
where 
\[
K(t) = [\Sigma(t) E(t)^\top + C(t)F(t)^\top]H(t)^{-1}.
\]

\section{Bounded Real Lemma with Partial Observation}\label{sec4}
In this section, we establish the partially observed bounded real lemma, which provides a necessary and sufficient condition for the system to satisfy the $H_\infty$ robustness constraint.

We consider the stochastic differential system \eqref{eq:3.1}
and the measured system output \[z(t) = Q(t)x(t),\] to quantify the system's sensitivity to external disturbance, we define the following perturbation operator and associated norm.

\begin{defn}\label{defn:4.1}
	For system \eqref{eq:3.1}, we define the linear operator $\mathcal{L} : \mathcal{V}^\mathbb{Y} \rightarrow L_\mathbb{F}^2 ([0,T];\mathbb{R}^{n})$ by
	\begin{equation*}
		(\mathcal{L}v)(t) := \widetilde{z}(t) = Q(t)[x^{x_0,0,v}(t) - x^{x_0,0,0}(t)].
	\end{equation*}
	
	The induced operator norm of $\mathcal{L}$, known as the system's $H_\infty$ norm is given by:
	\begin{equation*}
		\begin{aligned}
			\|\mathcal{L}\| &:= \sup_{\substack{v(\cdot) \neq 0 \\ v(\cdot) \in \mathcal{V}^{\mathbb{Y}}}} \frac{\|\widetilde{z}(t)\|_{[0,T]}}{\|v(t)\|_{[0,T]}}.
		\end{aligned}
	\end{equation*}
	where
	\begin{equation*}
		\begin{aligned}
			&\|v(t)\|_{[0,T]} = \Big({\mathbb{E}\int_0^T |v(t)|^2 \mathrm{d}t}\Big)^\frac{1}{2},\\
			&\|\widetilde{z}(t)\|_{[0,T]} = {\Big(\mathbb{E}\int_0^T |Q(t)[x^{x_0,0,v}(t) - x^{x_0,0,0}(t)]|^2 \mathrm{d}t \Big)}^\frac{1}{2}.
		\end{aligned}
	\end{equation*}
\end{defn}

With an initial condition $(0,x_0)$ for system \eqref{eq:3.1}, the robust performance index is defined as:
\begin{equation*}
	\begin{aligned}
		\min_{v(\cdot) \in \mathcal{V}^{\mathbb{Y}}} J_1(0,x_0;0,v&(\cdot)) :=\min_{v(\cdot) \in \mathcal{V}^{\mathbb{Y}}} \mathbb{E}\int_0^T [\gamma^2|v(t)|^2 - |z(t)|^2]\mathrm{d}t\\
		=&\min_{v(\cdot) \in \mathcal{V}^{\mathbb{Y}}} \mathbb{E}\int_0^T [\gamma^2|v(t)|^2 - |Q(t)x(t)|^2]\mathrm{d}t.
	\end{aligned}
\end{equation*}

Taking inspiration from \cite{Sun2022}, we use the orthogonal decomposition technique to decompose the cost functional $J_1(0,x_0;0,v(\cdot))$. Noting that the estimate error $\widetilde{x}(\cdot)$ is independent of disturbance $v(\cdot)$ and initial condition $(0,x_0)$. Moreover, with the discussion in Section \ref{sec3}, $\widehat{x}(\cdot)$ is orthogonal to $\widetilde{x}(\cdot)$, i.e. $\mathbb{E}\langle\widehat{x}(\cdot),\widetilde{x}(\cdot)\rangle = 0$. Then the cost functional $J_1(0,x_0;0,v(\cdot))$ can be decomposed as follows:
\begin{equation*}
	\begin{aligned}
		J_1(0,x_0;0,v(\cdot)) = \widehat{J}_1(x_0;v(\cdot)) + \widetilde{J}_1,
	\end{aligned}
\end{equation*}
where
\begin{equation*}
	\begin{aligned}
		&\widehat{J}_1(x_0;v(\cdot)) = \mathbb{E}\int_0^T \langle\gamma^2v(t),v(t)\rangle - \langle Q(t)^{\top}Q(t)\widehat{x}(t),\widehat{x}(t)\rangle \mathrm{d}t,\\
		&\widetilde{J}_1 = - \mathbb{E}\int_0^T \langle Q(t)^{\top}Q(t)\widetilde{x}(t),\widetilde{x}(t)\rangle \mathrm{d}t.\\
	\end{aligned}
\end{equation*}

Note that the stochastic optimal control problem corresponding to $\widehat{J}_1(x_0;v(\cdot))$ is a complete observed LQ problem, since the state estimation $\widehat{x}(\cdot)$ is adapted to filtration $\mathbb{Y}$, and $\widetilde{J}_1$ does not depend on the disturbance $v(\cdot)$, so we only need to solve the previous complete observed LQ problem. The aforementioned LQ problem is indefinite since the cost functional has the positive weighting matrix on the disturbance and the negative weighting matrix on the state.

For $\widehat{J}_1(x_0;v(\cdot))$, recall the filtering equation \eqref{eq:3.10}, we introduce the indefinite backward differential Riccati equation:
\begin{equation}\label{eq:4.1}
	\left\{
	\begin{aligned}
		&\dot{P}(t) + P(t)[A(t) + B(t)] + [A(t) + B(t)]^\top P(t)\\
		&\qquad - Q(t)^\top Q(t) - \gamma^{-2}P(t)B_1(t)B_1(t)^\top{P}(t) = 0,\\
		&P(T) = 0,
	\end{aligned}
	\right.
\end{equation}
and the backward ordinary differential equation(ODE):
\begin{equation}\label{eq:4.2}
	\left\{
	\begin{aligned}
		& \dot{\eta}(t) + [A(t) + B(t) - \gamma^{-2}B_1(t)B_1(t)^\top P(t)]^\top \eta(t)\\
		&\qquad \qquad \qquad \qquad \qquad \qquad \qquad + P(t)\sigma(t)= 0,\\
		& \eta(T) = 0.
	\end{aligned}
	\right.
\end{equation}

For $\widetilde{J}_1$, recall the estimate error equation \eqref{eq:3.11}. Correspondingly, we introduce a (backward) Riccati differential equation and an (backward) ordinary differential equation (ODE) as follow:
\begin{equation}\label{eq:4.3}
	\left\{
	\begin{aligned}
		& \dot{\Pi}(t) + \Pi(t)\mathcal{A}(t) + \mathcal{A}(t)^{\top}\Pi(t) -Q(t)^{\top}Q(t) = 0,\\
		& \Pi(T) = 0,
	\end{aligned}
	\right.
\end{equation}

\begin{equation}\label{eq:4.4}
	\left\{
	\begin{aligned}
		& \dot{\varphi}(t) + \mathcal{A}(t)^{\top}\varphi(t) = 0,\\
		& \varphi(T) = 0,
	\end{aligned}
	\right.
\end{equation}
where $\mathcal{A}(t)$ defined by \[\mathcal{A}(t) := A(t) - \Sigma(t)E(t)^{\top}H(t)^{-1}E(t) - C(t)F(t)^{-1}E(t).\]

\begin{rmk}\label{rmk:4.2}
	Assume that Assumptions \ref{ass:2.1}-\ref{ass:2.4} hold. By Theorem 7.2 in \cite{Yong1999}, Riccati equation \eqref{eq:4.3} admits a unique solution $\Pi(\cdot) \in C([0,T];\mathbb{S}^n)$. In addition, ODEs \eqref{eq:4.2} and \eqref{eq:4.4} admit unique solutions $\eta(\cdot) \in L^2_{\mathbb{Y}}([0,T];\mathbb{R}^n)$ and $\varphi(\cdot) \in L^2_{\mathbb{Y}}([0,T];\mathbb{R}^n)$, respectively.
\end{rmk}

Now we can present the Stochastic Bounded Real Lemma for stochastic differential systems with partial observation. The following result establishes a necessary and sufficient condition for the $H_\infty$ norm of the
system to be less than a prescribed disturbance attenuation level $\gamma$, based on the solvability of the Riccati equation \eqref{eq:4.1}. In addition, according to the solutions of equations \eqref{eq:4.2}, \eqref{eq:4.3} and \eqref{eq:4.4}, the observation-feedback representation of the worst-case disturbance and corresponding optimal value can be derived.

\begin{lem}\label{lem:4.3}
	Let Assumptions \ref{ass:2.1}-\ref{ass:2.4} hold. For a given disturbance attenuation level $\gamma > 0$, the system \eqref{eq:3.1} satisfies $\|\mathcal{L}\| < \gamma$ if and only if the Riccati equation \eqref{eq:4.1} admits a solution $P(\cdot) \in C([0,T];\mathbb{S}^n)$. Moreover, the corresponding worst-case disturbance in Definition \ref{defn:2.2} is given by the following observation-feedback form:
	\begin{equation}\label{eq:4.5}
		v^*(t) = -\gamma^{-2}B_1(t)^{\top}[P(t)\widehat{x}(t) + \eta(t)],
	\end{equation}
	where $\widehat{x}(t)$ evolves
	\begin{equation*}
		\left\{
		\begin{aligned}
			&\mathrm{d}\widehat{x}(t) = \Big\{[A(t) + B(t) - \gamma^{-2}B_1(t)B_1(t)^{\top}P(t)]\widehat{x}(t) \\
			&\qquad \quad- \gamma^{-2}B_1(t)B_1(t)^{\top}\eta(t) + \sigma(t)\Big\}\,\mathrm{d}t \\
			&\qquad + \Big[\Sigma(t)E(t)^{\top}H(t)^{-1} + C(t)F(t)^{-1}\Big] \\
			&\qquad \quad\cdot \Big\{\mathrm{d}y(t) - \Big[E(t)\widehat{x}(t) + \beta(t)\Big]\mathrm{d}t\Big\},\\
			&\widehat{x}(0) = \widehat{x}_0,
		\end{aligned}
		\right.
	\end{equation*}
	and the corresponding optimal value is expressed as:
	\begin{equation}\label{eq:4.6}
		\begin{aligned}
			J_1&(0,x_0;0,v^*(\cdot)) = \mathbb{E} \langle P(0)\widehat{x}_0,\widehat{x}_0 \rangle + 2\mathbb{E}\langle \eta(0),\widehat{x}_0 \rangle\\
			&+ \mathbb{E}\int_0^T \Big\{ 2\langle \eta(t),\sigma(t) \rangle - \langle \gamma^{-2} B_1(t)^\top \eta(t),B_1(t)^\top \eta(t) \rangle\\
			&\qquad + \langle P(t)[\Sigma(t) E(t)^\top H(t)^{\top,-1} + C(t)],\\
			&\qquad \quad [\Sigma(t)E(t)^\top H(t)^{\top,-1} + C(t)] \rangle\\
			&\qquad +\langle\Pi(t)[\Sigma(t)E(t)H(t)^{\top,-1}],[\Sigma(t)E(t)H(t)^{\top,-1}]\rangle \\
			&\qquad + \langle\Pi(t)D(t),D(t)\rangle \Big\} \mathrm{d}t.
		\end{aligned}
	\end{equation}
\end{lem}

\begin{proof}
	Let $\bar{x}(t) := x^{x_0,0,v}(t) - x^{x_0,0,0}(t)$, a straightforward calculation shows that $\bar{x}(t)$ satisfies the equation:
	\begin{equation}\label{eq:4.7}
		\left\{
		\begin{aligned}
			&\dot{\bar{x}}(t) = A(t)\bar{x}(t) + B(t)\widehat{\bar{x}}(t) + B_1(t)v(t),\\
			&\bar{x}(0) = 0.
		\end{aligned}
		\right.
	\end{equation}
	
	The state process $\bar{x}(\cdot)$ of system \eqref{eq:4.7} is entirely determined by the disturbance process $v(\cdot)$. Since $v(\cdot) \in \mathcal{V}^{\mathbb{Y}}$ is adapted to the filtration $\mathbb{Y}$, it follows that $\bar{x}(\cdot)$ is also adapted to $\mathbb{Y}$. Hence, $\widehat{\bar{x}}(t)=\bar{x}(t)$, and equation \eqref{eq:4.7} is equivalent to 
	
	\begin{equation*}
		\left\{
		\begin{aligned}
			&\dot{\bar{x}}(t) = [A(t) + B(t)]\bar{x}(t) + B_1(t)v(t),\\
			&\bar{x}(0) = 0.
		\end{aligned}
		\right.
	\end{equation*}
	
	From the definition of the operator norm $\left\|\mathcal{L}\right\|$ in Definition \ref{defn:4.1}, this can be viewed as an $H_\infty$ control problem for an ordinary deterministic linear system. According to \textit{Lemma 2.2} of \cite{Limebeer1994}, for system \eqref{eq:3.1}, the condition $\left\|\mathcal{L}\right\| < \gamma$ holds if and only if the Riccati equation \eqref{eq:4.1} admits a solution $P(\cdot) \in C([0,T];\mathbb{S}^n)$.
	
	Suppose that the Riccati equation \eqref{eq:4.3} admits a unique solution $\Pi(\cdot) \in C([0,T];\mathbb{S}^n)$, as well as ODEs \eqref{eq:4.2} and \eqref{eq:4.4} admit unique solutions $\eta(\cdot)$ and $\varphi(\cdot)$ in $L^2_{\mathbb{Y}}([0,T];\mathbb{R}^n)$. By the definition of the modified innovation process $\widehat{I}(t)$ in \eqref{eq:3.4}, the filtering equation \eqref{eq:3.10} can be written as follow:
	\begin{equation*}
		\left\{
		\begin{aligned}
			&\mathrm{d}\widehat{x}(t) = \Big\{[A(t) + B(t)]\widehat{x}(t) + B_1(t)v(t) + \sigma(t)\Big\}\,\mathrm{d}t \\
			&\qquad + \Big\{\Sigma(t)E(t)^{\top}F(t)^{\top,-1} + C(t)\Big\}\mathrm{d}\widehat{I}(t),\\
			&\widehat{x}(0) = \widehat{x}_0.
		\end{aligned}
		\right.
	\end{equation*}
	
	By applying It\^{o}'s formula to $\langle P(t)\widehat{x}(t) + 2\eta(t),\widehat{x}(t) \rangle$, we obtain
	\begin{equation*}
		\begin{aligned}
			\mathbb{E} &\langle P(T)\widehat{x}(T) + 2\eta(T),\widehat{x}(T) \rangle - \mathbb{E} \langle P(0)\widehat{x}_0 + 2\eta(0),\widehat{x}_0 \rangle \\
			= &\mathbb{E} \int_0^T \Big\{2\langle P(t)\{[A(t) + B(t)]\widehat{x}(t) + B_1(t)v(t) + \sigma(t)\},\widehat{x}(t) \rangle \\
			&\qquad + \langle P(t)[\Sigma(t)E(t)^{\top}F(t)^{\top,-1} + C(t)],\\
			&\qquad [\Sigma(t)E(t)^{\top}F(t)^{\top,-1} + C(t)] \rangle + \langle \dot{P}(t)\widehat{x}(t),\widehat{x}(t) \rangle\Big\} \mathrm{d}t\\
			+& \mathbb{E} \int_0^T \Big\{2\langle \eta(t),[A(t) + B(t)]\widehat{x}(t) + B_1(t)v(t) + \sigma(t) \rangle\\
			&\qquad + 2\langle \dot{\eta}(t),\widehat{x}(t) \rangle \Big\} \mathrm{d}t.
		\end{aligned}
	\end{equation*}
	
	Noting that $P(T) = 0,\eta(T) = 0$ and substituting for $\mathbb{E} \langle P(T)\widehat{x}(T) + 2\eta(T),\widehat{x}(T) \rangle$ in the cost functional $\widehat{J}_1(x_0;v(\cdot))$, then combining equation \eqref{eq:4.1} and \eqref{eq:4.2} and applying the method of completing the square, we obtain
	\begin{equation}\label{eq:4.8}
		\begin{aligned}
			\widehat{J}_1(x_0;&v(\cdot)) = \mathbb{E} \langle P(0)\widehat{x}_0,\widehat{x}_0 \rangle + 2\mathbb{E} \langle \eta(0),\widehat{x}_0 \rangle\\
			&+ \mathbb{E} \int_0^T \Big\{\langle \gamma^2\{v(t) + \gamma^{-2}B_1(t)^{\top}[P(t)\widehat{x}(t) + \eta(t)]\},\\
			&\qquad \{v(t) + \gamma^{-2}B_1(t)^{\top}[P(t)\widehat{x}(t) + \eta(t)]\} \rangle\\
			&\quad + \langle P(t)[\Sigma(t) E(t)^\top H(t)^{\top,-1} + C(t)],\\
			&\qquad [\Sigma(t)E(t)^\top H(t)^{\top,-1} + C(t)] \rangle\\
			&\quad + 2\langle \eta(t),\sigma(t) \rangle - \langle \gamma^{-2} B_1(t)^\top \eta(t),B_1(t)^\top \eta(t) \rangle\Big\} \mathrm{d}t.
		\end{aligned}
	\end{equation}
	
	Thus, the cost functional $\widehat{J}_1(x_0;v(\cdot))$ is minimized by the choice of $v^*(\cdot)$ of \eqref{eq:4.5}.
	
	Using the notation of $\mathcal{A}(t)$, the SDE \eqref{eq:3.11} can be written as:
	\begin{equation*}
		\left\{
		\begin{aligned}
			\mathrm{d}\widetilde{x}(t) =& \mathcal{A}(t)\widetilde{x}(t)\mathrm{d}t - \Big\{\Sigma(t)E(t)^{\top}H(t)^{\top,-1}\Big\}\mathrm{d}W(t) \\
			&+ D(t)\mathrm{d}\widetilde{W}(t),\\
			\widetilde{x}(0) = &\widetilde{x}_0.
		\end{aligned}
		\right.
	\end{equation*}
	
	Applying It\^{o}'s formula to $\langle \Pi(t)\widetilde{x}(t) + 2\varphi(t),\widetilde{x}(t) \rangle$ and substituting $\mathbb{E} \langle \Pi(T)\widetilde{x}(T) + 2\varphi(T),\widetilde{x}(T) \rangle$ in the cost functional $\widetilde{J}_1$, combining equation \eqref{eq:4.3} and \eqref{eq:4.4}, we obtain
	\begin{equation}\label{eq:4.9}
		\begin{aligned}
			\widetilde{J}_1& = \mathbb{E} \int_0^T \Big\{\langle \Pi(t)D(t),D(t) \rangle \\
			&+ \langle \Pi(t)[\Sigma(t)E(t)^{\top}H(t)^{\top,-1}],\Sigma(t)E(t)^{\top}H(t)^{\top,-1} \rangle\Big\} \mathrm{d}t.
		\end{aligned}
	\end{equation}
	
	Due to the cost functional $J_1(0,x_0;0,v(\cdot))$ can be written as the sum of $\widehat{J}_1(x_0;v(\cdot))$ and $\widetilde{J}_1$, and $\widetilde{J}_1$ does not depend on the disturbance $v(\cdot)$. Thus, minimizing $\widehat{J}_1(x_0;v(\cdot))$ with respect to $v^*(\cdot)$ also minimizes $J_1(0,x_0;0,v(\cdot))$. Finally, adding $\widehat{J}_1(x_0;v^*(\cdot))$ and $\widetilde{J}_1$ we get the optimal value \eqref{eq:4.6}.
\end{proof}

\section{The Solution of $H_2/H_\infty$ Control Problem}\label{sec5}
In this section, we establish a connection between the existence of a closed-loop Nash equilibrium point for the $H_2/H_\infty$ problem under partial information and the solvability of the corresponding Riccati equations. Moreover, an observation-based feedback representation for the solution to the $H_2/H_\infty$ problem is presented.

To state the main results of this paper, we introduce the cross-coupled Riccati equations as follows:
\begin{align}
	&\dot{P}_1 + P_1(A + B_2U) + (A + B_2U)^{\top}P_1 - Q^{\top}Q \nonumber\\
	&\qquad \qquad \qquad - U^{\top}U - \gamma^{-2}P_1 B_1 B_1^{\top} P_1 = 0, \label{eq:5.1}\\
	&\dot{P}_2 + P_2(A + B_1V) + (A + B_1V)^{\top}P_2 + Q^{\top}Q \nonumber\\
	&\qquad\qquad\qquad\qquad\qquad\quad - P_2 B_2 B_2^{\top} P_2 = 0, \label{eq:5.2}\\
	&P_1(T) = 0, P_2(T) = 0, \nonumber\\
	&V = -\gamma^{-2}B_1^{\top}P_1, U = -B_2^{\top}P_2. \nonumber
\end{align}

For the affine correction terms pertaining to the disturbance and control strategies, we introduce the following backward ODEs to characterize them:
\begin{align}
	&\dot{\eta}_1 + [A + B_2U - \gamma^{-2}B_1B_1^{\top}P_1]^{\top}\eta_1 + P_1(B_2U_0 + b) = 0, \label{eq:5.3}\\
	&\dot{\eta}_2 + [A + B_1V - B_2B_2^{\top}P_2]^{\top}\eta_2 + P_2(B_1V_0 + b) = 0, \label{eq:5.4}\\
	&\eta_1(T) = 0, \eta_2(T) = 0, \nonumber\\
	&V = -\gamma^{-2}B_1^{\top}P_1, U = -B_2^{\top}P_2;\nonumber\\
	&V_0 = -\gamma^{-2}B_1^{\top}\eta_1, U_0 = -B_2^{\top}\eta_2. \label{eq:5.5}
\end{align}

Furthermore, to express the corresponding optimal value, we also need to introduce the following standard Riccati equations:
\begin{align}
	&\dot{\Pi}_1 + \Pi_1\mathcal{A} + \mathcal{A}^{\top}\Pi_1 - Q^{\top}Q = 0, \label{eq:5.6}\\
	&\dot{\Pi}_2 + \Pi_2\mathcal{A} + \mathcal{A}^{\top}\Pi_2 + Q^{\top}Q = 0, \label{eq:5.7}\\
	&\Pi_1(T) = 0, \Pi_2(T) = 0, \nonumber
\end{align}
and the ODEs:
\begin{align}
	&\dot{\varphi}_1 + \mathcal{A}^{\top}\varphi_1 = 0, \label{eq:5.8}\\
	&\dot{\varphi}_2 + \mathcal{A}^{\top}\varphi_2 = 0, \label{eq:5.9}\\
	&\varphi_1(T) = 0, \varphi_2(T) = 0 \nonumber
\end{align}
where $\mathcal{A}(t)$ is defined by \[\mathcal{A}(t) := A(t) - \Sigma(t)E(t)^{\top}H(t)^{-1}E(t) - C(t)F(t)^{-1}E(t).\]

\begin{rmk}
	It can be readily observed that the solutions to equation \eqref{eq:5.8} and equation \eqref{eq:5.9} are identical, i.e. $\varphi_1(t) \equiv \varphi_2(t), \forall t \in [0,T]$. Under Assumptions \ref{ass:2.1}-\ref{ass:2.4}, it follows from Theorem 7.2 in \cite{Yong1999} that the ODE \eqref{eq:5.3} and the Riccati equation \eqref{eq:5.6} admit unique solutions $\eta_1(\cdot) \in L_{\mathbb{Y}}^2([0,T];\mathbb{R}^n)$ and $\Pi_1(\cdot) \in C([0,T];\mathbb{S}^n)$, respectively.
\end{rmk}

We now first address the system's $H_2$ optimality when the worst-case disturbance is imposed on system \eqref{eq:2.1}.
\begin{lem}\label{lem:5.2}
	Let Assumptions \ref{ass:2.1}-\ref{ass:2.4} hold. Suppose the disturbance process admits the following observation-feedback form:
	\begin{equation}\label{eq:5.10}
		v(\cdot) = V(\cdot)\widehat{x}(\cdot) + V_0(\cdot), \quad(V(\cdot),V_0(\cdot)) \in \mathcal{M}^2[0,T].
	\end{equation}
	
	Substituting \eqref{eq:5.10} into the state equation \eqref{eq:2.1}, we obtain the following differential system which involves not only the state process $x(\cdot)$ but also the conditional mathematical expectation $\widehat{x}(\cdot)$ of the state process:
	\begin{equation}\label{eq:5.11}
		\left\{
		\begin{aligned}
			\mathrm{d}x(t) =& \Big\{ A(t)x(t) + B_1(t)V(t)\widehat{x}(t) + B_2(t)u(t) \\
			&+ B_1(t)V_0(t) + b(t) \Big\} \mathrm{d}t \\
			&+ C(t)\mathrm{d}W(t) + D(t)\mathrm{d}\widetilde{W}(t),\\
			x(0)=&x_0.
		\end{aligned}  
		\right.
	\end{equation}
	
	The aforementioned system, in conjunction with the $H_2$ optimality criterion \[J_2(0,x_0;u(\cdot),v(\cdot)) = \mathbb{E}\int_0^T [|Q(t)x(t)|^2 + |N_1(t)u(t)|^2] \mathrm{d}t\] forms an LQ problem under partial information.
	
	Then the Riccati equations \eqref{eq:5.2} and \eqref{eq:5.7} admit unique solutions $P_2(\cdot) \in C([0,T];\mathbb{S}^n)$ and $\Pi_2(\cdot) \in C([0,T];\mathbb{S}^n)$, as well as ODEs \eqref{eq:5.4} and \eqref{eq:5.9} admit unique solutions $\eta_2(\cdot) \in L_{\mathbb{Y}}^2([0,T];\mathbb{R}^n)$ and $\varphi_2(\cdot) \in L_{\mathbb{Y}}^2([0,T];\mathbb{R}^n)$.
	
	Moreover, for the aforementioned partially observed LQ problem, there exists a unique optimal observation-feedback control of the form 
	\[
	u^*(t) = -B_2^{\top}(t)P_2(t)\widehat{x^*}(t) - B_2^{\top}(t)\eta_2(t),
	\]
	where $\widehat{x^*}(\cdot)$ is the optimal estimate state under $u^*$,
	and the corresponding optimal cost is given by
	\begin{equation}\label{eq:5.12}
		\begin{aligned}
			J_2(0,&x_0;u^*(\cdot),v(\cdot)) = \mathbb{E} \langle P_2(0)\widehat{x}_0,\widehat{x}_0 \rangle + 2\mathbb{E}\langle \eta_2(0),\widehat{x}_0 \rangle\\
			&+ \mathbb{E}\int_0^T \Big\{ 2\langle \eta_2(t),B_1(t)V_0(t) + b(t) \rangle \\
			&\quad - \langle B_2(t)^\top \eta_2(t),B_2(t)^\top \eta_2(t) \rangle\\
			&\quad + \langle P_2(t)[\Sigma(t) E(t)^\top H(t)^{\top,-1} + C(t)],\\
			&\qquad [\Sigma(t)E(t)^\top H(t)^{\top,-1} + C(t)] \rangle\\
			&\quad +\langle\Pi_2(t)[\Sigma(t)E(t)H(t)^{\top,-1}],[\Sigma(t)E(t)H(t)^{\top,-1}]\rangle \\
			&\quad + \langle\Pi_2(t)D(t),D(t)\rangle \Big\} \mathrm{d}t.
		\end{aligned}
	\end{equation}
\end{lem}

\begin{proof}
	It follows that this constitutes an LQ problem under partial information. According to the results presented in Section \ref{sec3}, the filtering equation transforms the original partially observed system into an equivalent fully observed system in terms of the conditional expectation $\widehat{x}(\cdot)$. By employing the orthogonal decomposition technique, the cost functional can be decomposed into two parts corresponding to $\widehat{x}(\cdot)$ and $\widetilde{x}(\cdot)$, respectively. Then, by reformulating the problem as a standard stochastic LQ problem with complete information in terms of $\widehat{x}(\cdot)$, the desired result follows from Theorem 3.5 in \cite{Sun2022}.
\end{proof}

The following theorem reveals a sufficient condition for the existence of a closed-loop Nash equilibrium to the $H_2/H_\infty$ problem under partial information.

\begin{thm}\label{thm:5.3}
	Let Assumptions \ref{ass:2.1}-\ref{ass:2.4} hold. Suppose the Riccati equations \eqref{eq:5.1}-\eqref{eq:5.2}, \eqref{eq:5.6}-\eqref{eq:5.7} admit solutions 
	\[
	(P_1(\cdot), P_2(\cdot), \Pi_1(\cdot), \Pi_2(\cdot)) \in C([0,T];\mathbb{S}^n)^4
	\]
	with corresponding feedback gains
	\[
	(U(\cdot),U_0(\cdot)) \in \mathcal{N}^2[0,T], \quad (V(\cdot),V_0(\cdot)) \in \mathcal{M}^2[0,T]
	\]
	and ODEs \eqref{eq:5.3}-\eqref{eq:5.4}, \eqref{eq:5.8}-\eqref{eq:5.9} admit unique solutions 
	\[
	(\eta_1(\cdot), \eta_2(\cdot), \varphi_1(\cdot), \varphi_2(\cdot)) \in L_{\mathbb{Y}}^2([0,T];\mathbb{R}^n)^4.
	\]
	
	Then the $H_2/H_\infty$ control problem defined in Definition \ref{defn:2.2} admits a closed-loop Nash equilibrium. More precisely, there exists a closed-loop Nash equilibrium $(u^*(\cdot),v^*(\cdot))$ for Problem \ref{pro:2.3}, and the optimal closed-loop control laws $(u^*(\cdot),v^*(\cdot))$ are given by:
	\begin{equation*}
		\begin{aligned}
			&u^*(\cdot) = U(\cdot)\widehat{x^*}(\cdot) + U_0(\cdot),\\
			&v^*(\cdot) = V(\cdot)\widehat{x^*}(\cdot) + V_0(\cdot),
		\end{aligned}
	\end{equation*}
	here $\widehat{x^*}(\cdot)$ is the optimal estimate process corresponding to system \eqref{eq:2.1} under the control pair $(u^*(\cdot),v^*(\cdot))$, and $(U(\cdot),U_0(\cdot);V(\cdot),V_0(\cdot)) \in \mathcal{N}^2[0,T] \times \mathcal{M}^2[0,T] $ is determined by \eqref{eq:5.5}.
\end{thm}
\begin{proof}
	Suppose that equations \eqref{eq:5.1}-\eqref{eq:5.9} are all solvable. Substituting $u(t) = U(t)\widehat{x}(t) + U_0(t)$ into state equation \eqref{eq:2.1}, we have the following closed-loop system:
	\begin{equation*}
		\left\{
		\begin{aligned}
			\mathrm{d}x(t) =& \Big\{ A(t)x(t) + B_2(t)U(t)\widehat{x}(t) + B_1(t)v(t) \\
			& + B_2(t)U_0(t) + b(t) \Big\} \mathrm{d}t \\
			& + C(t)\mathrm{d}W(t) + D(t)\mathrm{d}\widetilde{W}(t),\\
			x(0)=&x_0.
		\end{aligned}
		\right.
	\end{equation*}
	
	Using the conclusions from Section \ref{sec3}, we can obtain the following filtering equation and the estimate error equation:
	\begin{equation}\label{eq:5.13}
		\left\{
		\begin{aligned}
			&\mathrm{d}\widehat{x}(t) = \Big\{[A(t) + B_2(t)U(t)]\widehat{x}(t) + B_1(t)v(t) \\
			&\qquad + B_2(t)U_0(t) + b(t)\Big\}\,\mathrm{d}t \\
			&\qquad + \big[\Sigma(t)E(t)^{\top}H(t)^{-1} + C(t)F(t)^{-1}\big] \\
			&\qquad\quad \cdot \big\{\mathrm{d}y(t) - \big[E(t)\widehat{x}(t) + \beta(t)\big]\mathrm{d}t\big\},\\
			&\widehat{x}(0) = \widehat{x}_0,
		\end{aligned}
		\right.
	\end{equation}
	\begin{equation}\label{eq:5.14}
		\left\{
		\begin{aligned}
			\mathrm{d}\widetilde{x}(t) =& \Big\{A(t) - \Sigma(t)E(t)^{\top}H(t)^{-1}E(t) \\
			&\quad - C(t)F(t)^{-1}E(t)\Big\}\widetilde{x}(t)\mathrm{d}t\\
			-& \Big\{\Sigma(t)E(t)^{\top}H(t)^{-1}F(t)\Big\}\mathrm{d}W(t) + D(t)\mathrm{d}\widetilde{W}(t),\\
			\widetilde{x}(0) = &\widetilde{x}_0.
		\end{aligned}
		\right.
	\end{equation}
	
	We consider measured output $z(t)$ given by
	\begin{equation*}
		\begin{pmatrix}
			Q(t)x(t)\\
			N_1(t)[U(t)\widehat{x}(t) + U_0(t)]
		\end{pmatrix}.
	\end{equation*} 
	
	According to the definition of the cost functional $J_1$ and applying the orthogonal decomposition technique, we can decompose $J_1(0,x_0;U(t)\widehat{x}(\cdot) + U_0(\cdot),v(\cdot)) = \widehat{J}_1(x_0;U(t)\widehat{x}(\cdot) + U_0(\cdot),v(\cdot)) + \widetilde{J}_1$, where
	\begin{equation*}
		\begin{aligned}
			&\widehat{J}_1(x_0;U(t)\widehat{x}(\cdot) + U_0(\cdot),v(\cdot)) = \mathbb{E}\int_0^T \Big[\langle\gamma^2v(t),v(t)\rangle \\
			&\qquad \qquad- \langle[Q(t)^{\top}Q(t) + U(t)^{\top}U(t)]\widehat{x}(t),\widehat{x}(t)\rangle\Big]\mathrm{d}t,\\
			&\widetilde{J}_1 = -\mathbb{E}\int_0^T \langle Q(t)^{\top}Q(t)\widetilde{x}(t),\widetilde{x}(t)\rangle \mathrm{d}t.
		\end{aligned}
	\end{equation*}
	
	By using It\^{o}'s formula to $\langle P_1(t)\widehat{x}(t) + 2\eta_1(t),\widehat{x}(t)\rangle$ and $\langle \Pi_1(t)\widetilde{x}(t) + 2\varphi_1(t),\widetilde{x}(t)\rangle$, then simplify using \eqref{eq:5.1}, \eqref{eq:5.6} and completing the square, we obtain
	\begin{equation}\label{eq:5.15}
		\begin{aligned}
			J_1(0,&x_0;U(t)\widehat{x}(\cdot) + U_0(\cdot),v(\cdot)) \\
			&= \mathbb{E} \langle P_1(0)\widehat{x}_0,\widehat{x}_0 \rangle + 2\mathbb{E} \langle \eta_1(0),\widehat{x}_0 \rangle\\
			&+ \mathbb{E} \int_0^T \Big\{ 2\langle \eta_1(t), B_2(t)U_0(t) + b(t) \rangle \\
			& - \langle \gamma^{-2} B_1(t)^\top \eta_1(t),B_1(t)^\top \eta_1(t) \rangle\\
			& + \langle \gamma^2\{v(t) + \gamma^{-2}B_1(t)^{\top}[P_1(t)\widehat{x}(t) + \eta_1(t)]\},\\
			&\quad\{v(t) + \gamma^{-2}B_1(t)^{\top}[P_1(t)\widehat{x}(t) + \eta_1(t)]\} \rangle\\
			&+ \langle P_1(t)[\Sigma(t) E(t)^\top H(t)^{\top,-1} + C(t)],\\
			&\quad [\Sigma(t)E(t)^\top H(t)^{\top,-1} + C(t)] \rangle\\
			& + \langle \Pi_1(t)[\Sigma(t)E(t)^{\top}H(t)^{\top,-1}],\Sigma(t)E(t)^{\top}H(t)^{\top,-1} \rangle \\
			& + \langle \Pi_1(t)D(t),D(t) \rangle \Big\} \mathrm{d}t.
		\end{aligned}
	\end{equation}
	
	Since the coefficient of the quadratic term associated with the disturbance is positive, taking $v(t) = v^*(t) = V(t)\widehat{x^*}(t) + V_0(t) = -\gamma^{-2}B_1(t)^{\top}[P_1(t)\widehat{x^*}(t) + \eta_1(t)]$ and $u(t) = u^*(t) = U(t)\widehat{x^*}(t) + U_0(t)$, it follows immediately that for any $v(\cdot) \in \mathcal{V}^{\mathbb{Y}}$,
	\begin{equation}\label{eq:5.16}
		\begin{aligned}
			J_1(0,x_0;&u^*(\cdot),v^*(\cdot)) \\
			&= J_1(0,x_0;U(t)\widehat{x^*}(\cdot) + U_0(\cdot),V(\cdot)\widehat{x^*}(\cdot) + V_0(\cdot))\\
			&\leq J_1(0,x_0;U(\cdot)\widehat{x}(\cdot) + U_0(\cdot),v(\cdot)).
		\end{aligned}
	\end{equation}
	
	That is, $v^*(\cdot) = V(\cdot)\widehat{x^*}(\cdot) + V_0(\cdot)$ is the worst-case disturbance that minimizes cost functional $J_1(0,x_0;U(\cdot)\widehat{x}(\cdot) + U_0(\cdot),v(\cdot))$. Moreover, letting $x(0) = 0, b(\cdot) = C(\cdot) = D(\cdot) = 0$ and using Lemma \ref{lem:4.3}, the robustness condition $\|\mathcal{L}\| < \gamma$ holds.
	
	Substituting $v(t) = V(t)\widehat{x}(t) + V_0(t)$ into the state equation and considering the minimization of the cost functional $J_2(0,x_0;u(\cdot),V(\cdot)\widehat{x}(\cdot) + V_0(\cdot))$, we observe that this in fact constitutes a partially observed LQ problem. Taking $u(t) = u^*(t) = U(t)\widehat{x^*}(t) + U_0(t) = -B_2^{\top}(t)[P_2(t)\widehat{x^*}(t) + \eta_2(t)]$ and $v(t) = v^*(t) = V(t)\widehat{x^*}(t) + V_0(t)$, then for any $u(\cdot) \in \mathcal{U}^{\mathbb{Y}}$, we can easily obtain 
	\begin{equation}\label{eq:5.17}
		\begin{aligned}
			J_2(0,x_0;&u^*(\cdot),v^*(\cdot)) \\
			&= J_2(0,x_0;U(\cdot)\widehat{x^*}(\cdot) + U_0(\cdot),V(\cdot)\widehat{x^*}(\cdot) + V_0(\cdot))\\
			&\leq J_2(0,x_0;u(\cdot),V(\cdot)\widehat{x}(\cdot) + V_0(\cdot)).
		\end{aligned}
	\end{equation}
	by Lemma \ref{lem:5.2}. That is, $u^*(\cdot) = U(\cdot)\widehat{x^*}(\cdot) + U_0(\cdot)$ is the optimal control when the worst-case disturbance is imposed on system \eqref{eq:2.1}. The existence of the Nash equilibrium $(u^*(\cdot), v^*(\cdot))$ is an immediate result by combining \eqref{eq:5.16} and \eqref{eq:5.17}, and the closed-loop Nash equilibrium strategy $(U(\cdot),U_0(\cdot);V(\cdot),V_0(\cdot))$ clearly satisfies \eqref{eq:5.5}.
\end{proof}

Next, we state a theorem that provides a necessary condition for the solvability of the stochastic $H_2/H_\infty$ control problem with partial information in terms of the existence of solutions to the associated Riccati equations.

\begin{thm}\label{thm:5.4}
	Suppose that the $H_2/H_\infty$ control problem with partial observation (as defined in Definition \ref{defn:2.2}) admits a solution $(u^*(\cdot),v^*(\cdot))$ with observation-feedback form in \eqref{eq:2.3}, then the cross-coupled Riccati equations \eqref{eq:5.1}-\eqref{eq:5.2} admit a solution $(P_1(\cdot),P_2(\cdot)) \in C([0,T];\mathbb{S}^n)^2$.
\end{thm}

\begin{proof}
	On the one hand, assume the $H_2/H_\infty$ problem with partial observation is solvable, then condition (1) in Definition \ref{defn:2.2} holds. It then follows from Lemma \ref{lem:4.3} that the following Riccati equation possesses a unique solution $P_1 \in C([0,T];\mathbb{S}^n)$:
	\begin{equation*}
		\left\{
		\begin{aligned}
			&\dot{P}_1 + P_1(A + B_2U) + (A + B_2U)^{\top}P_1 - Q^{\top}Q - U^{\top}U \\
			&\qquad\qquad\qquad\qquad\qquad\qquad - \gamma^{-2}P_1 B_1 B_1^{\top} P_1 = 0,\\
			&P_1(T) = 0.
		\end{aligned}
		\right.
	\end{equation*}
	
	Furthermore, the worst-case disturbance is given by 
	\begin{equation*}
		v^*(t) = -\gamma^{-2}B_1(t)^{\top}[P_1(t)\widehat{x^*}(t) + \eta_1(t)],
	\end{equation*}
	where $\eta_1 \in C([0,T];\mathbb{R}^n)$ is the solution to ODE:
	\begin{equation*}
		\left\{
		\begin{aligned}
			&\dot{\eta}_1 + [A + B_2U - \gamma^{-2}B_1B_1^{\top}P_1]^{\top}\eta_1 + P_1(B_2U_0 + b) = 0,\\
			&\eta_1(T) = 0.
		\end{aligned}
		\right.
	\end{equation*}
	
	Hence, $(V,V_0)$ is determined by
	\begin{equation*}
		(V,V_0) = (-\gamma^{-2}B_1^{\top}P_1,-\gamma^{-2}B_1^{\top}\eta_1) \in \mathcal{M}^2[0,T].
	\end{equation*}
	
	On the other hand, we consider the worst-case disturbance $v^* = V\widehat{x^*} + V_0$ imposed on system \eqref{eq:2.1}. By Lemma \ref{lem:5.2}, the next two equations admit unique solutions $P_2 \in C([0,T];\mathbb{S}^n)$ and $\eta_2 \in C([0,T];\mathbb{R}^n)$:
	\begin{equation*}
		\left\{
		\begin{aligned}
			&\dot{P}_2 + P_2(A + B_1V) + (A + B_1V)^{\top}P_2 + Q^{\top}Q \\
			&\qquad\qquad\qquad\qquad\qquad - P_2 B_2 B_2^{\top} P_2 = 0,\\
			&P_2(T) = 0.
		\end{aligned}
		\right.
	\end{equation*}
	\begin{equation*}
		\left\{
		\begin{aligned}
			&\dot{\eta}_2 + [A + B_1V - B_2B_2^{\top}P_2]^{\top}\eta_2 + P_2(B_1V_0 + b) = 0,\\
			&\eta_2(T) = 0.
		\end{aligned}
		\right.
	\end{equation*}
	and the corresponding optimal observation-feedback control is
	\begin{equation*}
		u^*(t) = -B_2(t)^{\top}[P_2(t)\widehat{x^*}(t) + \eta_2(t)].
	\end{equation*}
	
	It is clear that $(U,U_0)$ is determined by
	\begin{equation*}
		(U,U_0) = (-B_2^{\top}P_2,-B_2^{\top}\eta_2) \in \mathcal{N}^2[0,T].
	\end{equation*}
	
	Consequently, the cross-coupled Riccati equations \eqref{eq:5.1}-\eqref{eq:5.2} admit solutions with coupling terms $U = -B_2^{\top}P_2, V = -\gamma^{-2}B_1^{\top}P_1$.
\end{proof}

\section{An UAV Numerical Example}\label{sec6}
Referring to References \cite{Qi2025,Hui2024,Hasanlu2025,Ming2023,Shi2018}, this section applies the proposed partial observation stochastic $H_{2}/H_{\infty}$ control framework to the longitudinal flight control of a quadrotor unmanned aerial vehicle (UAV). The objective is to validate the effectiveness of the observer-based feedback controller in dealing with system biases and suppressing the worst-case disturbance derived from differential game theory.

\subsection{System Modeling}
We consider the linearized longitudinal dynamic model of a quadrotor UAV. The physical parameters are adopted from \cite{Shi2018}: mass $m=2.0$ kg, gravitational acceleration $g=9.81$ m/s$^2$, arm length $l=0.2$ m, and the moment of inertia along the y-axis $I_y=0.005$ kg$\cdot$m$^2$.

The state vector is defined as $x(t) = [z(t), \dot{z}(t), \theta(t), \dot{\theta}(t)]^\top$, representing the vertical position, vertical velocity, pitch angle, and pitch rate, respectively. The control input $u(t) = [U_1(t), U_2(t)]^\top$ consists of the normalized total thrust and pitch torque. The system dynamics are described by the stochastic differential equation:
\begin{equation*}
	\mathrm{d}x(t) = [A x(t) + B_1 v(t) + B_2 u(t) + b(t)] \mathrm{d}t + C \mathrm{d}W(t).
\end{equation*}

The system matrices are given by:
\begin{equation*}
	A = \begin{bmatrix} 
		0 & 1 & 0 & 0 \\
		0 & 0 & 0 & 0 \\
		0 & 0 & 0 & 1 \\
		0 & 0 & 0 & 0 
	\end{bmatrix}, 
	B_2 = \begin{bmatrix} 
		0 & 0 \\
		1/m & 0 \\
		0 & 0 \\
		0 & 1/I_y
	\end{bmatrix}, 
	B_1 = \begin{bmatrix} 
		0 & 0 \\
		0.1 & 0 \\
		0 & 0 \\
		0 & 0.1 
	\end{bmatrix}.
\end{equation*}

The affine term $b(t) = [0, -0.05, 0, 0.02]^\top$ represents a constant model mismatch, such as unmodeled gravity offset and aerodynamic trim imbalance.

To rigorously model the stochastic environment, we assume the system and the observation are driven by a standard 2-dimensional Brownian motion $W(t)$. The noise intensity matrix $C$ is chosen as:
\begin{equation*}
	C = 0.01 \times \begin{bmatrix} 
		1 & 0 \\ 1 & 0 \\ 0 & 1 \\ 0 & 1 
	\end{bmatrix}.
\end{equation*}

We consider an observation equation incorporating a sensor bias term $\beta(t)$:
\begin{equation*}
	\mathrm{d}y(t) = [E x(t) + \beta(t)] \mathrm{d}t + F \mathrm{d}W(t),
\end{equation*}
where the controller has access only to the position $z$ and pitch angle $\theta$ measurements:
\begin{equation*}
	E = \begin{bmatrix} 
		1 & 0 & 0 & 0 \\
		0 & 0 & 1 & 0 
	\end{bmatrix}, \quad
	F = 0.05 I_2.
\end{equation*}
The affine term $\beta(t) = [0.02, 0.01]^\top$ simulates a constant bias in the sensors.

\subsection{Controller Design}

The design objective is to satisfy the $H_\infty$ disturbance attenuation level $\gamma = 0.8$ while minimizing the $H_2$ performance. The weighting matrices for the cost functional are chosen as $Q = \text{diag}(1, 0.1, 1, 0.1)$ and $N_1 = I_2$.

According to Theorem \ref{thm:5.3}, the optimal control strategy is $u^*(t) = -B_2^\top P_2(t) \widehat{x^*}(t) - B_2^{\top}\eta_2(t)$. Simultaneously, to verify the robustness, we apply the worst-case disturbance strategy derived from the Nash equilibrium of the differential game:
\begin{equation*}
	v^*(t) = -\gamma^{-2} B_1^\top P_1(t) \widehat{x^*}(t) - \gamma^{-2}B_1^{\top}\eta_1(t).
\end{equation*}

This strategy represents the most aggressive adversarial attack attempting to maximize the performance cost. The state estimate $\widehat{x}(t)$ is generated by the filter:
\begin{equation*}
	\begin{aligned}
		\mathrm{d}\widehat{x}(t) = &[A\widehat{x}(t) + B_1 v^*(t) + B_2 u^*(t) + b(t)]\mathrm{d}t \\
		&+ K(t)\{\mathrm{d}y(t) - [E\widehat{x}(t) + \beta(t)]\mathrm{d}t\},
	\end{aligned}
\end{equation*}
where $K(t) = \Sigma(t)E^{\top}H^{-1} + CF^{-1}$.

\subsection{Simulation Results}

The simulation is conducted over a time horizon $T=20\,\mathrm{s}$ with a time step of $0.01\,\mathrm{s}$. The initial state is set to $x_0 = [0.5, 0, 0.1, 0]^\top$, representing a deviation from the hover equilibrium.

\begin{figure}[htbp]
	\centering
	\includegraphics[width=\linewidth]{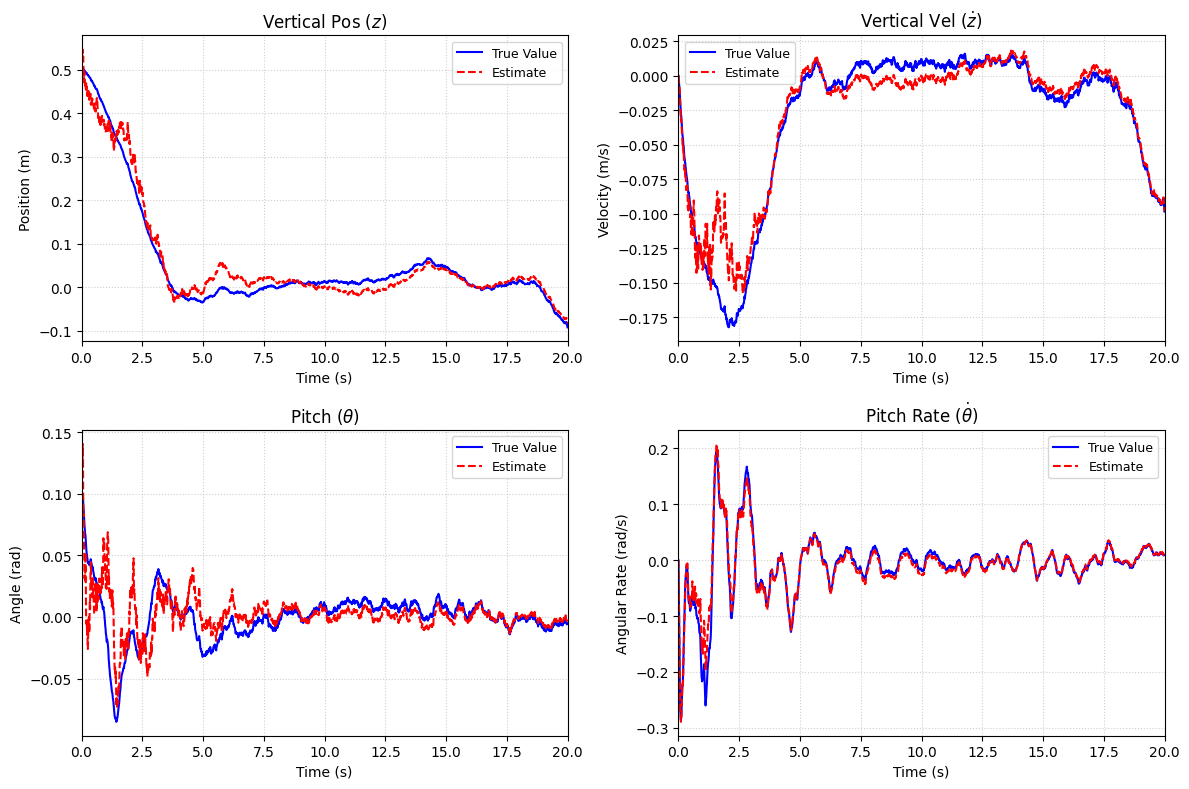}
	\caption{Trajectories of true states $x(t)$ (blue solid lines) and estimated states $\widehat{x}(t)$ (red dashed lines) under the proposed controller and disturbance.}
	\label{fig:uav_worst_case}
\end{figure}

The simulation results are presented in Fig. \ref{fig:uav_worst_case}. The blue solid lines represent the true system states, while the red dashed lines depict the estimated states generated by the filter. It is observed that the state $\dot{z}$ exhibits deviations near the end of the simulation period ($t \approx 20\,\mathrm{s}$). This is a typical terminal effect in finite-horizon optimal control problems. Since the terminal cost is set to zero ($P(T)=0$), the control gains decay to zero as $t \to T$, reducing the control authority against the persistent non-homogeneous term $b(t)$.

Based on the simulation results, the effectiveness of the controller designed in this paper is demonstrated in the following two aspects.

\begin{itemize}
	\item \textbf{Robustness against Worst-Case disturbance:} As shown in Fig. \ref{fig:uav_worst_case}, when $t=1\,\mathrm{s}$, due to the worst-case disturbance $v^*(t)$, the system state trajectory (such as vertical velocity $\dot{z}$) has obvious mutation and deviation. However, thanks to the immediate response of controller $u^*(t)$, this deviation was quickly suppressed in a small bounded range and did not diverge. When the disturbance is removed for $t>5\,\mathrm{s}$, the system state quickly converges to the equilibrium point under the adjustment of the controller, which demonstrates the robust stability and anti-interference ability of the closed-loop system under partial observation. 
	\item \textbf{Estimation Performance:} Fig. \ref{fig:uav_worst_case} illustrates the state estimation performance of the closed-loop system, where the blue solid lines represent the true states and the red dashed lines depict the estimated states. It can be observed that, despite the presence of stochastic noises $W(t)$ and time-varying sensor bias $\beta(t)$, the filter achieves precise tracking of the system states. Particularly for the unmeasurable channels: vertical velocity $\dot{z}$ and pitch rate $\dot{\theta}$, the estimates converge rapidly to the true values after a short transient period. This demonstrates that the filter not only effectively filters out measurement noise but also successfully compensates for the systematic offsets induced by $b(t)$ and $\beta(t)$, thereby validating the integrity of the proposed control strategy under partial observation.
\end{itemize}

To quantitatively evaluate the $H_\infty$ performance of the closed-loop system, we computed the energy gain ratio from condition (1) in Definition \ref{defn:2.2}. The control input \( u(t) \) adopts a feedback control law based on the estimated state \( \widehat{x}(t) \):
\(
u(t) = -B_2^{\top} P_2(t) \widehat{x}(t) - B_2^{\top}\eta_2(t).
\)
To rigorously test the system's robustness, the simulation constructs a piecewise worst-case disturbance signal $v(t)$ designed to deteriorate the system's performance. During the initial triggering phase ($t \in [0, 1)$), a constant step disturbance $v(t) = [1.0, 1.0]^\top$ is applied to force the system away from its equilibrium state, thereby exciting its transient response. Subsequently, in the worst-case phase ($t \in [1, 5)$), a disturbance strategy derived from $H_\infty$ differential game theory is implemented to maximize the system's deterioration, which is formulated as
\(v(t) = -\gamma^{-2} B_1^\top P_1(t) \widehat{x}(t) -\gamma^{-2}B_1^{\top}\eta_1(t)\).
Finally, during the free evolution phase ($t \ge 5$), the external disturbance is removed ($v(t) = 0$) to observe the asymptotic stabilization process of the control system in coordination with the estimator. Numerical integration of the simulation data yields a ratio of:
\begin{equation*}
	\|L\| = \frac{\|\widetilde{z}(t)\|_{[0,T]}}{\|v^*(t)\|_{[0,T]}} = 0.366 < \gamma = 0.8,
\end{equation*}
this value is strictly less than the prescribed disturbance attenuation level. This result numerically validates that the closed-loop system satisfies the condition $\left\|L\right\| < \gamma$, thereby confirming the dissipativity of the system in the energetic sense.

In conclusion, the simulation results show that the proposed observer-based $H_2/H_\infty$ control strategy not only solves the problem of information loss caused by partial observations, but also effectively resists the worst-case disturbance and system deviation, and verifies the applicability of the theoretical derivation in the actual flight control scene.

\section{Conclusion}\label{sec7}
This paper addressed the mixed $H_{2}/H_{\infty}$ control problem for linear stochastic systems under partial observation. By formulating a nonzero-sum stochastic differential game, we derived a unified closed-loop solution. Key contributions include establishing a Stochastic Bounded Real Lemma via orthogonal decomposition and proving that the existence of a closed-loop Nash equilibrium is closely related to the solvability of cross-coupled Riccati differential equations (CRDEs). A ``separation-like'' principle was obtained, expressing the optimal control and worst-case disturbance as linear feedback of the state estimate. Numerical simulations of a quadrotor UAV validated the controller's effectiveness in disturbance suppression and cost minimization. Future work may extend this framework to infinite horizon or mean-field systems.


 




\vfill

\end{document}